\documentclass[oneside,english]{amsart}
\usepackage[T1]{fontenc}
\usepackage[latin9]{inputenc}
\usepackage{geometry}
\usepackage{amsthm}
\usepackage{amssymb}
\usepackage{graphicx}

\makeatletter
\numberwithin{equation}{section}
\numberwithin{figure}{section}
\theoremstyle{plain}
\newtheorem{thm}{\protect\theoremname}
\theoremstyle{plain}
\newtheorem{prop}[thm]{\protect\propositionname}
\theoremstyle{remark}
\newtheorem*{rem*}{\protect\remarkname}
\theoremstyle{plain}
\newtheorem{lem}[thm]{\protect\lemmaname}

\makeatother

\usepackage{babel}
\providecommand{\lemmaname}{Lemma}
\providecommand{\propositionname}{Proposition}
\providecommand{\remarkname}{Remark}
\providecommand{\theoremname}{Theorem}

\begin{document}
\title{unexpected logarithmic identities and other surprises}
\author{C. Vignat}
\address{Department of Mathematics, Tulane University, New Orleans, USA and
LSS, CentraleSupelec, Universit\'{e} Paris-Sud Orsay, France}
\begin{abstract}
This is a journey through integrals of involutions and surprising
consequences of the Lagrange inversion theorem. On the way, we meet
unexpected logarithmic identities, hypergeometric functions with a
linear regime and other mysterious objects. This study was inspired
by some results from the fascinating article \cite{Holroyd}.
\end{abstract}

\maketitle

\section{\label{sec:Introduction}Introduction}

We study here the suspiciously simple identity
\begin{equation}
\log\left[1-\sum_{n\ge1}\left(an\right)_{n-1}\frac{w^{n}}{n!}\right]=-\sum_{n\ge1}\left(an+1\right)_{n-1}\frac{w^{n}}{n!},\label{eq:log1}
\end{equation}
for a parameter $a>0$
and over the range $\vert w\vert\le\frac{\left(a+1\right)^{a+1}}{a^{a}},$
and where $\left(a\right)_{n}$  is the notation for the Pochhammer symbol $\frac{\Gamma\left(a+n\right)}{\Gamma\left(a\right)}.$

It will be shown that this identity can be obtained as a special
case of the classical Rothe-Hagen identity. But there is more: substituting
the variable $w$ with $x^{a}-x^{a+1}$ produces, now for $\vert x\vert\le1,$
the identity
\[
\log\left[1-\sum_{n\ge1}\left(an\right)_{n-1}\frac{\left(x^{a}-x^{a+1}\right)^{n}}{n!}\right]=-\sum_{n\ge1}\left(an+1\right)_{n-1}\frac{\left(x^{a}-x^{a+1}\right)^{n}}{n!}
\]
as expected, with the subtlety that over the interval $x\in\left[\frac{a}{a+1},1\right],$
this identity reduces to 
\[
x=x,
\]
meaning that, for $x\in\left[\frac{a}{a+1},1\right],$
\begin{align*}
\log\left[1-\sum_{n\ge1}\left(an\right)_{n-1}\frac{\left(x^{a}-x^{a+1}\right)^{n}}{n!}\right] & =x
\end{align*}
and
\[
-\sum_{n\ge1}\left(an+1\right)_{n-1}\frac{\left(x^{a}-x^{a+1}\right)^{n}}{n!}=x.
\]

How can such functions be related by a simple logarithm ? And how
can they coincide with the identity function over a whole interval
? 

Before we start explaining these strange identities, here is another
piece of the puzzle: a two parameters $0<a<b$ version of (\ref{eq:log1})
is
\[
-\log\left[1-\sum_{n\ge1}\left(\frac{a}{b-a}n\right)_{n-1}\frac{\left(x^{a}-x^{b}\right)^{n}}{n!}\right]=\sum_{n\ge1}\left(\frac{a}{b-a}n+1\right)_{n-1}\frac{\left(x^{a}-x^{b}\right)^{n}}{n!},\,\,\vert x\vert<1.
\]
It does not seem much deeper than (\ref{eq:log1}); however, the limit
case $\epsilon\to0$ of this identity with parameters $a=1,b=1+\epsilon$
produces
\[
-\log\left(1-\sum_{n\ge1}\frac{\left(n-1\right)^{n-1}}{n!}z^{n}\right)=\sum_{n\ge1}\frac{n^{n-1}}{n!}z^{n},
\]
an identity known to anyone who is familiar with the Lambert and the
Cayley tree functions.

The remaining of this article is dedicated to providing insight into
these identities, the functions involved and some of their integrals.

\section{\label{sec:sec2}Involutions}

An easy way to build involutions over $\left[0,1\right]$, i.e. functions
that satisfy 
\[
f\left(f\left(x\right)\right)=x,\,\,0\le x\le1,
\]
is as follows: start from a continuous unimodal function $\phi:\left[0,1\right]\to\mathbb{R}$,
say increasing from $0$ to $\phi\left(x_{0}\right)$ on $\left[0,x_{0}\right]$
and decreasing from $\phi\left(x_{0}\right)$ to $0$ on $\left[x_{0},1\right],$
and define $f\left(x\right)$ as the unique solution to the equation
\[
\phi\circ f\left(x\right)=\phi\left(x\right),\,\,x\ne x_{0}
\]
and
\[
f\left(x_{0}\right)=x_{0}.
\]

A typical example of this construction appears in \cite{Holroyd}
in the case $\phi\left(x\right)=-x\log x,\,\,x\in\left[0,1\right]$.
\begin{center}
\includegraphics[scale=0.2]{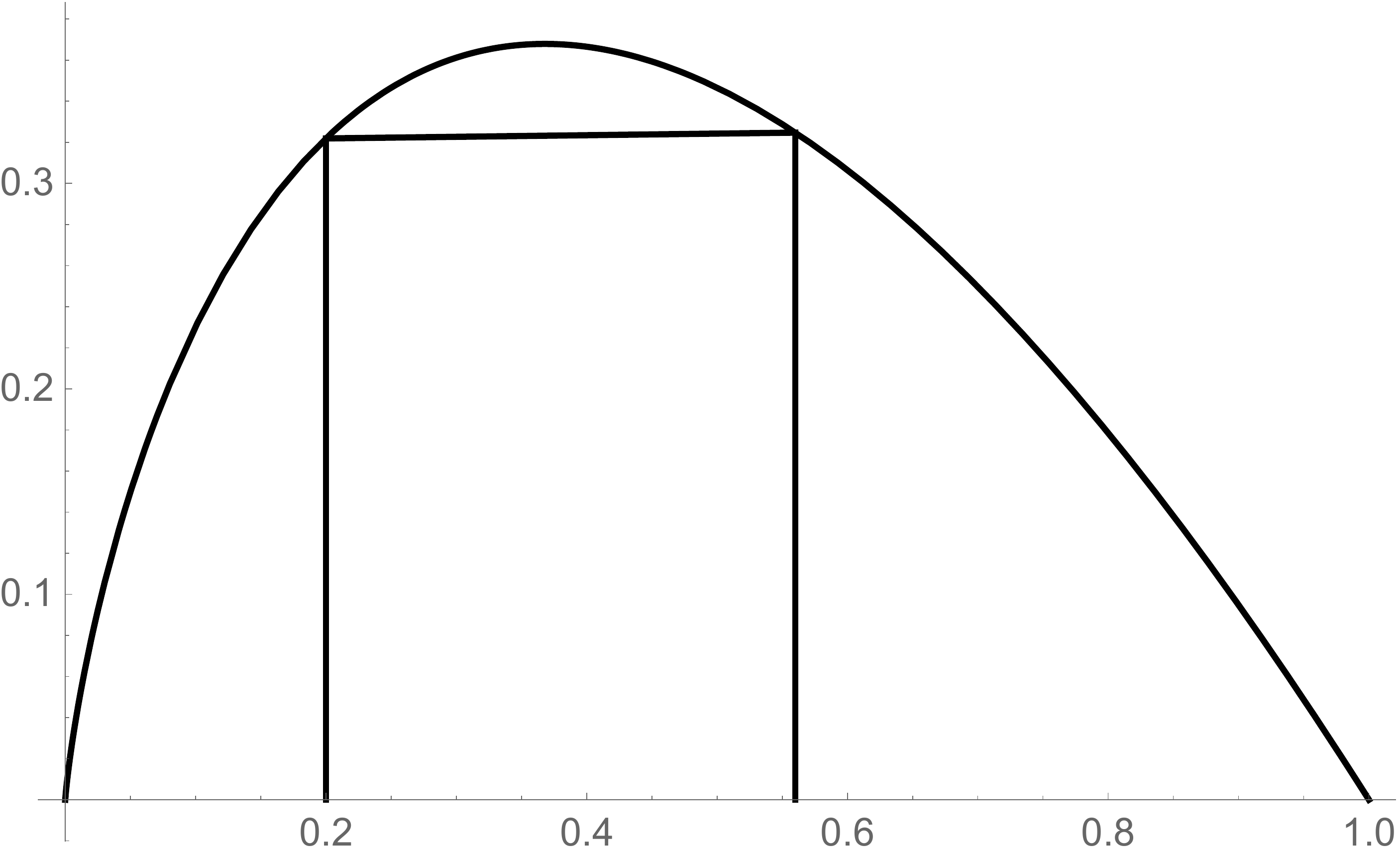}
\end{center}
Back to the general case, let us denote $l$ and $r$ the invertible
functions defined as the restrictions of $\phi$ on the intervals
$\left[0,x_{0}\right]$ and $\left[x_{0},1\right]$ respectively.
Then the function $f$ satisfies
\begin{equation}
f\left(x\right)=\begin{cases}
r^{-1}\circ\phi\left(x\right), & 0\le x\le x_{0},\\
l^{-1}\circ\phi\left(x\right), & x_{0}\le x\le1.
\end{cases}\label{eq:f-1}
\end{equation}
In the next sections, we'll study the involution $f$ in the case
where the inverse of the functions $r$ and $l$ are computed using
Lagrange inversion theorem. This technique provides the inverse $r^{-1}$
of a function $r$ in the neighborhood of a point $x=a$, assuming
that $r'\left(a\right)\ne0,$ in the form of the power series
\[
r^{-1}\left(x\right)=a+\sum_{n\ge1}c_{n}\frac{\left(x-r\left(a\right)\right)^{n}}{n!}
\]
with the coefficients
\[
c_{n}=\lim_{x\to a}\frac{d^{n-1}}{dx^{n-1}}\left(\frac{x-a}{r\left(x\right)-r\left(a\right)}\right)^{n}.
\]
In the case where $r'\left(1\right)\ne0$ and where the power series
expansion for $r^{-1}$ at $z=1$ converges on the domain $\left[0,x_{0}\right]$,
we will make use of the auxiliary function $g\left(x\right)$ defined
as \footnote{in the case $l'\left(0\right)\ne0$, we would choose an expansion
of $l^{-1}$ at $a=0$ and define rather $g\left(x\right)=l^{-1}\circ r\left(x\right),\thinspace\thinspace0\le x\le1.$}
\begin{equation}
g\left(x\right)=r^{-1}\circ\phi\left(x\right),\,\,0\le x\le1,\label{eq:g}
\end{equation}
noticing that
\begin{equation}
g\left(x\right)=\begin{cases}
f\left(x\right) & 0\le x\le x_{0}\\
x & x_{0}\le x\le1.
\end{cases}\label{eq:g2}
\end{equation}
The fact that the identity function appears in this construction suggests
a possible hint to the mysterious linear behavior described in Section \ref{sec:Introduction}.
Before we can confirm that this construction is related to our
problem - it obviously is - let us extend this study to the computation
of integrals.

\section{Integrals}

Assume that we want to evaluate the integral 
\[
I=\int_{0}^{1}f\left(x\right)dx
\]
with the function $f$ an involution as defined in the previous section.
Instead of computing both inverses $r^{-1}$ and $l^{-1}$ as required
by (\ref{eq:f-1}), we prefer to exploit the symmetry due to the involutive
property of $f:$ as illustrated in the examples below, the integral
\[
J=\int_{0}^{1}g\left(x\right)dx
\]
often turns out to be easier to evaluate than $I$ itself. 

Our main result in this section is the following proposition that
provides an identity between both integrals, allowing to compute the
integral $I$ - and in fact a generalization of it - in terms of the
easier $J.$
\begin{prop}
\label{prop:main}For $h$ a monotone, differentiable function defined
over $\left[0,1\right]$, and with the notation $h\left(f\left(1\right)\right)=\lim_{x\to1}h\left(f\left(x\right)\right),$
\begin{equation}
\int_{0}^{1}h\left(f\left(x\right)\right)h'\left(x\right)dx=2\int_{0}^{1}h\left(g\left(x\right)\right)h'\left(x\right)dx-h\left(1\right)\left(h\left(1\right)-h\left(f\left(1\right)\right)\right).\label{eq:general integral}
\end{equation}
Moreover, 
\begin{equation}
\int_{0}^{x_{0}}h\left(f\left(x\right)\right)h'\left(x\right)dx=\int_{0}^{1}h\left(g\left(x\right)\right)h'\left(x\right)dx+\frac{1}{2}\left(h^{2}\left(x_{0}\right)-h^{2}\left(1\right)\right).\label{eq:integral0x0}
\end{equation}
\end{prop}

In the simple case $h\left(x\right)=x,$ identities (\ref{eq:general integral})
and (\ref{eq:integral0x0}) have a simple geometric interpretation:
the following figure shows the graphs of both functions $f$ (the
quarter circle shaped curve) and $g$ (the top eighth circle shaped
curve over $\left[0,\frac{3}{4}\right]$ followed by the linear curve
over $\left[\frac{3}{4},1\right]$) in the special case $\phi\left(x\right)=x^{3}-x^{4},$
for which $x_{0}=\frac{3}{4}.$
\begin{center}
\includegraphics[scale=0.1]{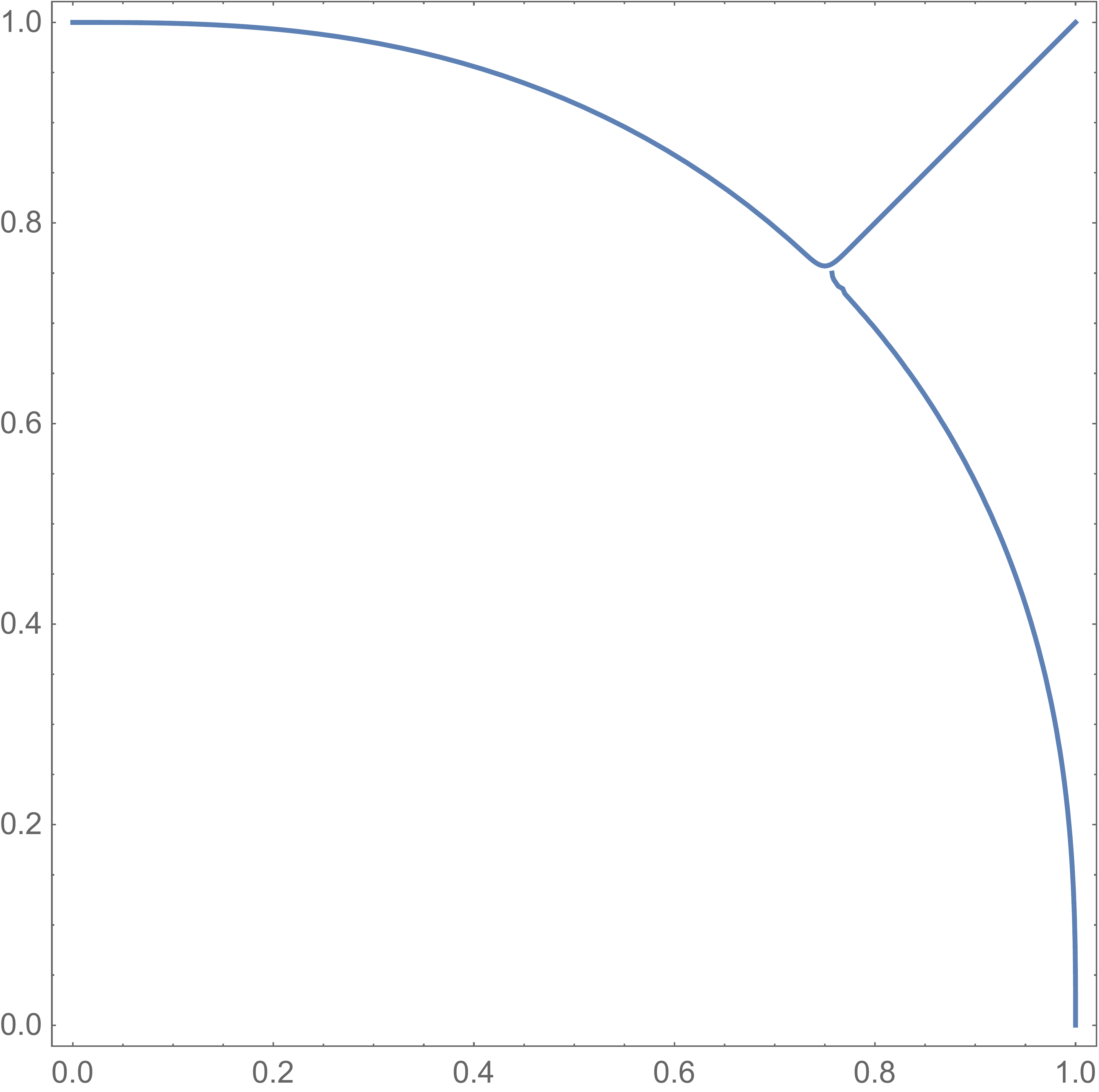}
\par\end{center}

Visualizing on this graph the different areas $I_{1},I_{2},J_{1}$
and $J_{2}$ that appear in the proof provides a straightforward geometric
interpretation of Proposition \ref{prop:main}.

The application of this result, together with the computation of functions
$f$ and $g$ in the special case $\phi\left(x\right)=x^{a}-x^{b},\,\,0<a<b$
are addressed in the following sections; the corresponding solutions
are denoted as $f_{a,b}$ and $g_{a,b};$ in the particular case $b=a+1$,
we use the shortcuts $f_{a}$ for $f_{a,a+1}$ and $g_{a}$ for $g_{a,a+1}$.

\section{Some special cases}

In this section we apply the construction of involutions using the
Lagrange inversion technique as described in Section \ref{sec:sec2}
to the parametrized function $\phi\left(x\right)=x^{a}-x^{b}$ with
$0<a<b.$

\subsection{the case $b=a+1$}

In the case $\phi\left(x\right)=x^{a}-x^{a+1},$ we have the following
results, with the notation $\left(a\right)_{n}=\frac{\Gamma\left(a+n\right)}{\Gamma\left(a\right)}$
for the Pochhammer symbol.
\begin{prop}
\label{prop:faga}The function $g_{a}$ defined by (\ref{eq:g}) has
the power series expansion
\begin{equation}
g_{a}\left(x\right)=1-\sum_{n\ge1}\left(an\right)_{n-1}\frac{\left(x^{a}\left(1-x\right)\right)^{n}}{n!},\thinspace\thinspace0\le x\le1\label{eq:Fa-1}
\end{equation}
and coincides with the function $f_{a}$ over the interval $\left[0,x_{0}\right].$
It coincides with the identity function on the complementary interval
$\left[x_{0},1\right].$ 

\end{prop}

\begin{rem*}
For integer values of $a,$ the function $g_{a}$ is a generalized
hypergeometric function
\[
g_{a}\left(x\right)=\frac{1}{a+1}+\frac{a}{a+1}\thinspace_{a}F_{a-1}\left(\begin{array}{c}
-\frac{1}{a+1},\frac{1}{a+1},\frac{2}{a+1},\dots,\frac{a-1}{a+1}\\
\frac{1}{a},\frac{2}{a},\dots,\frac{a-1}{a}
\end{array};\frac{\left(a+1\right)^{a+1}}{a^{a}}x^{a}\left(1-x\right)\right).
\]
Its linear behavior on the interval on $\left[x_{0},1\right]$
may thus appear as a surprise. However, this linear behavior appears
in disguise in the specialization $\nu=a+1,\mu=-1,z=\frac{1}{x}$
of Entry 5.2.13.30 in \cite{Prudnikov} ($a$ may not be integer here, and notice that $\left(0\right)_{-1}=-1$ )
\[
\sum_{n\ge0}\frac{\Gamma\left(n\nu+\mu-1\right)}{\Gamma\left(n\nu-n+\mu\right)}\frac{z^{n}}{n!}=\frac{y^{\mu-1}}{\mu-1},\,\,z=\frac{y-1}{y^{\nu}},\,\,\vert z\vert<\vert\frac{\left(\nu-1\right)^{\nu-1}}{\nu^{\nu}}\vert.
\]
\end{rem*}

Let us now apply Lagrange inversion formula to the function 
\[
\log\circ r_{a}^{-1}=\left(r_{a}\circ\exp\right)^{-1};
\]

this produces the unexpected identity
\begin{prop}
\label{prop:Prop2}For $a>0,$

\begin{equation}
-\log\left[1-\sum_{n\ge1}\left(an\right)_{n-1}\frac{w^{n}}{n!}\right]=\sum_{n\ge1}\left(an+1\right)_{n-1}\frac{w^{n}}{n!},\thinspace\thinspace\vert w\vert\le\frac{\left(a+1\right)^{a+1}}{a^{a}}.\label{eq:-log}
\end{equation}
\end{prop}

Not surprisingly, this identity can be identified as the specialization
of a more classical identity.
\begin{prop}
\label{prop:RotheHagen}Identity (\ref{eq:-log}) is equivalent to
a special case of the Rothe-Hagen identity \cite[3.146]{Gould}
\begin{equation}
\sum_{k=0}^{n}\binom{x+kz}{k}\binom{y-kz}{n-k}\frac{p+kq}{\left(x+kz\right)\left(y-kz\right)}=\frac{p\left(x+y-nz\right)}{x\left(x+y\right)\left(y-nz\right)}\binom{x+y}{n}\label{eq:RotheHagen}
\end{equation}
with the specialization $x=a,\thinspace\thinspace y=\left(a+1\right)n-1,\thinspace\thinspace z=a+1,\thinspace\thinspace p=a$
and $q=a+1.$
\end{prop}

\subsection{the general case $0<a<b$}

We now address the case $\phi_{a,b}\left(x\right)=x^{a}-x^{b}.$ The
function $g_{a,b}$ can be computed as follows.
\begin{prop}
\label{prop:Prop8}The function $g_{a,b}\left(x\right)$ satisfies
\begin{equation}
g_{a,b}^{b-a}\left(x\right)=1-\sum_{n\ge1}\left(\frac{a}{b-a}n\right)_{n-1}\frac{\left(x^{a}-x^{b}\right)^{n}}{n!},\thinspace\thinspace0\le x\le1.\label{eq:Fabb-a}
\end{equation}
\end{prop}

Applying now Lagrange inversion formula to 
\[
\log r_{a,b}^{-1}=\left(r_{a,b}\circ\exp\right)^{-1}
\]
produces a power series expansion for the function $-\log g_{a,b}.$
\begin{prop}
\label{prop:Prop7}With $b-a>0,$ we have
\begin{equation}
-\log g_{a,b}\left(x\right)=\frac{1}{b-a}\sum_{n\ge1}\left(\frac{a}{b-a}n+1\right)_{n-1}\frac{\left(x^{a}-x^{b}\right)^{n}}{n!},\,\,0\le x\le1.\label{eq:Fab}
\end{equation}
This extends identity (\ref{eq:-log}) to
\[
-\log\left[1-\sum_{n\ge1}\left(\frac{a}{b-a}n\right)_{n-1}\frac{w^{n}}{n!}\right]=\sum_{n\ge1}\left(\frac{a}{b-a}n+1\right)_{n-1}\frac{w^{n}}{n!},\,\,\vert w\vert<\left(\frac{a^{a}}{b^{b}}\right)^{\frac{1}{b-a}}.
\]
\end{prop}

\section{Integrals}
\begin{prop}
\label{prop:Prop6}The integral of $f_{a}$ over $\left[0,1\right]$
is evaluated as
\begin{align}
\int_{0}^{1}f_{a}\left(x\right)dx & =\frac{1}{1+a}\left[1-\frac{a\pi}{a+1}\cot\frac{a\pi}{a+1}\right]\label{eq:intfa}
\end{align}
whereas
\[
\int_{0}^{x_{0}}f_{a}\left(x\right)dx=\frac{1}{2\left(1+a\right)^{2}}\left[1+a+a^{2}-a\pi\cot\frac{a\pi}{a+1}\right].
\]
\end{prop}

We provide two proofs of the first result: one as a straightforward
application of our main result Proposition \ref{prop:main}, and another
based on the expression of the integral (\ref{eq:intfa}) as a double
integral, a clever technique that we borrow from \cite{Andrews}.

In the general case, the inversion technique used above in the case
$b=a+1$ extends to the general case using the following observation.
\begin{lem}
Assume that $\frac{b'}{a'}=\frac{b}{a}$ then
\[
g_{a,b}^{b-a}\left(y^{\frac{1}{a}}\right)=g_{a',b'}^{b'-a'}\left(y^{\frac{1}{a'}}\right).
\]
As a consequence, with $a'=\frac{a}{b-a}$ and $b'=\frac{b}{b-a},$
\[
g_{a',a'+1}\left(x\right)=g_{a,b}^{b-a}\left(x^{\frac{1}{b-a}}\right).
\]
\end{lem}

\begin{proof}
These identities can be checked directly from the expression (\ref{eq:Fabb-a}),
or deduced from the remark (see \cite{Holroyd}) that the function
$s\left(x\right)=\left(f_{a,b}\left(x^{\gamma}\right)\right)^{\frac{1}{\gamma}}$
satisfies the equation
\[
s^{a\gamma}\left(x\right)-s^{b\gamma}\left(x\right)=x^{a\gamma}-x^{b\gamma}
\]
so that it coincides with $f_{a\gamma,b\gamma}.$

From these identities, we deduce
\end{proof}
\begin{prop}
\label{prop:Prop10}We have the following evaluations
\[
\int_{0}^{1}g_{a,b}^{b-a}\left(x\right)x^{b-a-1}dx=\frac{1}{b-a}\left[1-\frac{a}{2b^{2}}\left(b+\left(b-a\right)\pi\cot\left(\frac{\pi a}{b}\right)\right)\right],
\]
\begin{align*}
\int_{0}^{1}f_{a,b}^{b-a}\left(x\right)x^{b-a-1}dx & =\frac{2-b+a-\frac{a}{b}}{b-a}-\frac{a}{b^{2}}\pi\cot\left(\frac{\pi a}{b}\right).
\end{align*}
\end{prop}

\section{The limit Lambert case}

The function $\phi\left(x\right)=-x\log x,\,\,x\in\left[0,1\right]$
is shown to be the limit case of $\phi_{1,1+\epsilon}$ above as $\epsilon\to0.$
It is a famous case of Lagrange inversion and is related to the Lambert
and the Cayley tree functions.
\begin{prop}
\label{prop:Lambertfunction}In the limit Lambert case, we have
\[
g\left(w\right)=1-\sum_{n\ge1}\frac{\left(n-1\right)^{n-1}}{n!}\left(-w\log w\right)^{n},\,\,0\le w\le1
\]
and 
\[
\log g\left(w\right)=\sum_{n\ge1}\frac{\left(-n\right)^{n-1}}{n!}\left(w\log w\right)^{n}=W_{0}\left(w\log w\right),
\]
where $W_{0}$ is the principal branch of the Lambert function. 

The identity
\[
-\log\left(1-\sum_{n\ge1}\frac{\left(n-1\right)^{n-1}}{n!}z^{n}\right)=\sum_{n\ge1}\frac{n^{n-1}}{n!}z^{n}
\]
is the limit case of (\ref{eq:-log}).
\end{prop}

Notice that we recover the identity 
\[
\log g\left(w\right)=\log w,\,\,w\ge\frac{1}{e},
\]
as a well-known property of the Lambert function $W_{0}$. 
\begin{prop}
\label{prop:Lambertintegrals}We have the following integrals
\[
\int_{0}^{1}g\left(w\right)dw=1-\sum_{n\ge1}\frac{\left(n-1\right)^{n-1}}{\left(n+1\right)^{n+1}}=0.659495\dots
\]
\[
\int_{0}^{1}f\left(w\right)dw=2-2\sum_{n\ge1}\frac{\left(n-1\right)^{n-1}}{\left(n+1\right)^{n+1}}=1.31899\dots
\]
\[
\int_{0}^{1}\frac{-\log g\left(w\right)}{w}dw=\frac{\pi^{2}}{6}
\]
and
\[
\int_{0}^{1}\frac{-\log f\left(x\right)}{x}dx=\frac{\pi^{2}}{3}.
\]
\end{prop}

\section{Proofs}

\subsection{Proof of Proposition \ref{prop:main}}

Denote
\[
I_{1}=\int_{0}^{x_{0}}h\left(f\left(x\right)\right)h'\left(x\right)dx,\thinspace\thinspace I_{2}=\int_{x_{0}}^{1}h\left(f\left(x\right)\right)h'\left(x\right)dx.
\]
Changing variable $y=f\left(x\right)$ and integrating by parts in
$I_{1}$ produces
\begin{align}
I_{1} & =-\left[h\left(y\right)h\left(f\left(y\right)\right)\right]_{x_{0}}^{1}+\int_{x_{0}}^{1}h'\left(y\right)h\left(f\left(y\right)\right)dy\nonumber \\
 & =h^{2}\left(x_{0}\right)-h\left(f\left(1\right)\right)h\left(1\right)+I_{2}\label{eq:eq1}
\end{align}
Considering now
\[
J_{1}=\int_{0}^{x_{0}}h\left(g\left(x\right)\right)h'\left(x\right)dx,\,\,J_{2}=\int_{x_{0}}^{1}h\left(g\left(x\right)\right)h'\left(x\right)dx,
\]
we have, by (\ref{eq:g2}), $J_{1}=I_{1}$ and 
\[
J_{2}=\int_{x_{0}}^{1}h\left(x\right)h'\left(x\right)dx=\frac{1}{2}\left[h^{2}\left(1\right)-h^{2}\left(x_{0}\right)\right].
\]
Moreover, define
\[
J=J_{1}+J_{2}=\int_{0}^{1}h\left(g\left(x\right)\right)h'\left(x\right)dx.
\]
We deduce
\[
h^{2}\left(x_{0}\right)=I_{1}-I_{2}+\bar{h}\left(1\right)h\left(1\right)=h^{2}\left(1\right)-2J_{2}=h^{2}\left(1\right)-2\left(J-I_{1}\right)
\]
so that
\begin{equation}
I_{1}+I_{2}=2J-h^{2}\left(1\right)+h\left(1\right)h\left(f\left(1\right)\right).\label{eq:eq2}
\end{equation}
Next, $I_{1}$ is deduced by solving the linear system that consists
of the two equations (\ref{eq:eq1}) and (\ref{eq:eq2}).

\subsection{Proof of Proposition \ref{prop:faga}}

Let us compute the inverse function $r_{a}^{-1}$ using Lagrange inversion
formula: at $w=1,$ the function
\[
w\mapsto r_{a}\left(w\right)=z=w^{a}\left(1-w\right)
\]
can be inverted since $\frac{dr_{a}}{dw}_{\vert w=1}\ne0.$ The coefficients
of the series expansion of its inverse
\[
w=r_{a}^{-1}\left(z\right)=1+\sum_{k\ge1}c_{n}\frac{z^{n}}{n!}
\]
are computed as
\begin{align*}
c_{n} & =\lim_{w\to1}\frac{d^{n-1}}{dw^{n-1}}\left(\frac{w-1}{\phi_{a,2}\left(w\right)-\phi_{a,2}\left(1\right)}\right)^{n}=\lim_{w\to1}\frac{d^{n-1}}{dw^{n-1}}\left(\frac{w-1}{w^{a}\left(1-w\right)}\right)^{n}\\
 & =\left(-1\right)^{n}\lim_{w\to1}\frac{d^{n-1}}{dw^{n-1}}\left(w^{-an}\right)=-\left(an\right)_{n-1}
\end{align*}
with $\left(a\right)_{n}=\frac{\Gamma\left(a+n\right)}{\Gamma\left(a\right)}.$
We deduce
\begin{equation}
r_{a}^{-1}\left(z\right)=1-\sum_{k\ge1}\left(an\right)_{n-1}\frac{z^{n}}{n!},\thinspace\thinspace0\le z\le\phi_{a}\left(x_{0}\right)\label{eq:ra-1}
\end{equation}
and
\begin{equation}
f_{a}\left(x\right)=1-\sum_{n\ge1}\left(an\right)_{n-1}\frac{\left(x^{a}\left(1-x\right)\right)^{n}}{n!},\thinspace\thinspace0\le x\le x_{0}.\label{eq:fa(x)}
\end{equation}
The series in (\ref{eq:fa(x)}) is convergent over $\mathbb{R}:$
let us define the function
\begin{equation}
g_{a}\left(x\right)=1-\sum_{n\ge1}\left(an\right)_{n-1}\frac{\left(x^{a}\left(1-x\right)\right)^{n}}{n!},\thinspace\thinspace0\le x\le1.\label{eq:Fa}
\end{equation}

\subsection{Proof of Proposition \ref{prop:RotheHagen}}

Noticing that $\left(0\right)_{-1}=-1$ allows to rewrite the desired
identity as
\[
-\log\left[-\sum_{n\ge0}\left(an\right)_{n-1}\frac{w^{n}}{n!}\right]=\sum_{n\ge1}\left(an+1\right)_{n-1}\frac{w^{n}}{n!}.
\]
Taking the derivative with respect to $w$ on both sides produces
\[
\sum_{n\ge0}\left(an+a+1\right)_{n}\frac{w^{n}}{n!}=-\frac{\sum_{n\ge0}\left(an+a\right)_{n}\frac{w^{n}}{n!}}{\sum_{n\ge0}\left(an\right)_{n-1}\frac{w^{n}}{n!}}
\]
so that we need to check the convolution identity
\[
\left(\sum_{n\ge0}\left(an+a+1\right)_{n}\frac{w^{n}}{n!}\right)\left(\sum_{n\ge0}\left(an\right)_{n-1}\frac{w^{n}}{n!}\right)=-\sum_{n\ge0}\left(an+a\right)_{n}\frac{w^{n}}{n!}
\]
or equivalently
\[
\sum_{k\ge0}\frac{\left(ak+a+1\right)_{k}}{k!}\frac{\left(an-ak\right)_{n-k-1}}{\left(n-k\right)!}=-\frac{\left(an+a\right)_{n}}{n!}.
\]
This is indeed identity (\ref{eq:RotheHagen}) with the specialization
$x=a,\thinspace\thinspace y=\left(a+1\right)n-1,\thinspace\thinspace z=a+1,\thinspace\thinspace p=a$
and $q=a+1.$

\subsection{Proof of Proposition \ref{prop:Prop2}}

The coefficients of the expansion are now computed as
\begin{align*}
c_{n} & =\lim_{w\to0}\frac{d^{n-1}}{dw^{n-1}}\left(\frac{w^{n}}{e^{naw}\left(1-e^{w}\right)^{n}}\right).
\end{align*}
Since
\[
\left(\frac{w}{e^{w}-1}\right)^{n}e^{-naw}=\sum_{k\ge0}\frac{B_{k}^{\left(n\right)}\left(-na\right)}{k!}w^{k}
\]
is the generating function for the higher-order Bernoulli polynomials
$B_{k}^{\left(n\right)}\left(z\right)$, we deduce
\[
c_{n}=B_{n-1}^{\left(n\right)}\left(-na\right).
\]
These special values of the higher-order Bernoulli polynomials appear
in \cite[Ch.6]{Norlund} as
\[
B_{n-1}^{\left(n\right)}\left(x\right)=\left(x-1\right)\dots\left(x-n+1\right)
\]
so that
\[
c_{n}=\left(-1\right)^{n}\left(-na-1\right)\dots\left(-na-n+1\right)=-\left(an+1\right)_{n-1}
\]
and we obtain the power series expansion
\[
\log r_{a}^{-1}=\sum_{k\ge1}\left(an+1\right)_{n-1}\frac{z^{n}}{n!},\thinspace\thinspace0\le z\le\phi_{a}\left(x_{0}\right).
\]

\subsection{Proof of Proposition \ref{prop:Prop6}}

Consider $h\left(x\right)=x$ in identity (\ref{eq:general integral}).
The right-hand side integral is computed using (\ref{eq:Fa}) as 
\[
\int_{0}^{1}g_{a}\left(x\right)dx=\int_{0}^{1}\left(1-\sum_{n\ge1}\left(an\right)_{n-1}\frac{\left(x^{a}\left(1-x\right)\right)^{n}}{n!}\right)dx=1-\sum_{n\ge1}\frac{\left(an\right)_{n-1}}{n!}\int_{0}^{1}\left(x^{a}\left(1-x\right)\right)^{n}dx
\]
and evaluating the beta integral produces
\begin{align*}
\int_{0}^{1}g_{a}\left(x\right)dx & =1-\sum_{n\ge1}\frac{\left(an\right)_{n-1}}{n!}\frac{\Gamma\left(an+1\right)\Gamma\left(n+1\right)}{\Gamma\left(an+n+2\right)}=\\
 & =1-\sum_{n\ge1}\frac{\Gamma\left(an+1\right)\Gamma\left(an+n-1\right)}{\Gamma\left(an+n+2\right)\Gamma\left(an\right)}=1-\frac{a}{2\left(1+a\right)}\left(1-\frac{\pi}{a+1}\cot\frac{\pi}{a+1}\right).
\end{align*}
Finally, as $h\left(1\right)\left(h\left(1\right)-\bar{h}\left(1\right)\right)=1,$
\begin{align*}
\int_{0}^{1}f_{a}\left(x\right)dx & =2\int_{0}^{1}g_{a}\left(x\right)dx-1=1-\frac{a}{1+a}\left(1-\frac{\pi}{a+1}\cot\frac{\pi}{a+1}\right)\\
 & =\frac{1}{a+1}-\frac{a\pi}{\left(a+1\right)^{2}}\cot\frac{a\pi}{a+1}.
\end{align*}
Moreover
\begin{align*}
\int_{0}^{x_{0}}f_{a}\left(x\right)dx & =1-\frac{a}{2\left(1+a\right)}\left(1-\frac{\pi}{a+1}\cot\frac{\pi}{a+1}\right)+\frac{1}{2}\left(\frac{a^{2}}{\left(1+a\right)^{2}}-1\right)\\
 & =\frac{1}{2\left(1+a\right)^{2}}\left[1+a+a^{2}-a\pi\cot\frac{a\pi}{a+1}\right].
\end{align*}

\subsection{A proof of Proposition \ref{prop:Prop6} borrowed from \cite{Andrews}}

We first restate, for completeness, Andrews, Eriksson, Petrov and
Romik's elegant solution \cite{Andrews} to the evaluation of 
\[
I_{a,b}=\int_{0}^{1}-\frac{\log f_{a,b}\left(x\right)}{x}dx=\frac{\pi^{2}}{3ab}.
\]
Rewrite $I_{a,b}$ as the double integral
\[
I_{a,b}=\int_{0}^{1}\frac{dx}{x}\int_{f_{a,b}\left(x\right)}^{1}\frac{dy}{y}=\iint_{\mathcal{D}}\frac{dxdy}{xy}
\]
over the domain $\mathcal{D}=\left\{ 0\le x\le1,f_{a,b}\left(x\right)\le y\le1\right\} .$
Dividing this domain into two equal subdomains and change variables
$\left[\begin{array}{c}
x\\
y
\end{array}\right]\to\left[\begin{array}{c}
x\\
t=\frac{y}{x}
\end{array}\right]$ so that $dxdy\to xdxdt$ and
\[
I_{a,b}=2\iint_{\mathcal{D}'}\frac{dxdt}{tx}
\]
over the new domain $\mathcal{D}'=\left\{ 0\le t\le1,\left(\frac{1-t^{a}}{1-t^{b}}\right)^{\frac{1}{b-a}}\le x\le1\right\} $
produces
\begin{align*}
I_{a,b} & =2\int_{0}^{1}\int_{\left(\frac{1-t^{a}}{1-t^{b}}\right)^{\frac{1}{b-a}}}^{1}\frac{dx}{x}\frac{dt}{t}=-2\int_{0}^{1}\log\left(\frac{1-t^{a}}{1-t^{b}}\right)^{\frac{1}{b-a}}dt\\
 & =\frac{2}{b-a}\left[\int_{0}^{1}\log\left(1-t^{b}\right)dt-\int_{0}^{1}\log\left(1-t^{a}\right)dt\right].
\end{align*}
Substituting $x=t^{b}$ in the first integral and $x=t^{a}$ in the
second provides the desired result.

This approach is now used to produce another proof of (\ref{eq:intfa})
as follows: denote 
\[
I_{a}=\int_{0}^{1}f_{a,a+1}\left(x\right)dx
\]
and rewrite it as the double integral
\[
I_{a}=\int_{0}^{1}\int_{0}^{f_{a,a+1}\left(x\right)}dydx=2\iint_{\mathcal{D}'}xdxdt
\]
with
\[
\mathcal{D}'=\left\{ 0\le t\le1,0\le x\le\frac{1-t^{a}}{1-t^{a+1}}\right\} 
\]
so that
\[
I_{a}=2\int_{0}^{1}\frac{1}{2}x^{2}\left(t\right)dt=\int_{0}^{1}\left(\frac{1-t^{a}}{1-t^{a+1}}\right)^{2}dt.
\]
This integral is evaluated using the change of variable $t=e^{z}$
producing
\[
I_{a}=\int_{0}^{\infty}\left(\frac{\sinh\left(\frac{az}{2}\right)}{\sinh\left(\frac{a+1}{2}z\right)}\right)^{2}dz.
\]
This is the special case $b=\frac{a}{2},\thinspace\thinspace c=\frac{a+1}{2}$
of Entry 2.4.4.2 in \cite{Prudnikov} 
\[
\int_{0}^{\infty}\left(\frac{\sinh bx}{\sinh cx}\right)^{2}dx=\frac{1}{2c}-\frac{\pi b}{2c^{2}}\cot\left(\frac{\pi b}{c}\right)
\]
so that 
\[
I_{a}=\frac{1}{1+a}-\frac{a\pi}{\left(a+1\right)^{2}}\cot\frac{a\pi}{a+1}.
\]

\subsection{Proof of Proposition \ref{prop:Prop7}}

The coefficients of the series expansions of $r_{a,b}^{-1}$ are
\begin{align*}
c_{n} & =\lim_{w\to0}\frac{d^{n-1}}{dw^{n-1}}\left(\frac{w^{n}}{e^{naw}\left(1-e^{w\left(b-a\right)}\right)^{n}}\right).
\end{align*}
Denoting $c=b-a$ and expanding
\[
\frac{w^{n}}{e^{naw}\left(1-e^{wc}\right)^{n}}=\left(-\frac{1}{c}\right)^{n}\frac{\left(cw\right)^{n}}{\left(e^{wc}-1\right)^{n}}e^{-naw}=\left(-\frac{1}{c}\right)^{n}\sum_{k\ge0}\frac{B_{k}^{\left(n\right)}\left(-n\frac{a}{c}\right)}{k!}\left(wc\right)^{k}
\]
produces

\begin{align*}
c_{n} & =\frac{\left(-1\right)^{n}}{c}B_{n-1}^{\left(n\right)}\left(-n\frac{a}{c}\right)=-\frac{\left(n\frac{a}{b-a}+1\right)_{n-1}}{b-a},
\end{align*}
so that, for $0\le x\le x_{0},$
\[
\frac{1}{b-a}\sum_{n\ge1}\left(\frac{a}{b-a}n+1\right)_{n-1}\frac{\left(x^{a}-x^{b}\right)^{n}}{n!}=-\log f_{a,b}\left(x\right).
\]

\subsection{Proof of Proposition \ref{prop:Prop8}}

Replacing $a$ with $\frac{a}{b-a}$ in (\ref{eq:-log}) produces
\[
\sum_{n\ge1}\left(\frac{a}{b-a}n+1\right)_{n-1}\frac{\left(x^{\frac{a}{b-a}}-x^{\frac{b}{b-a}}\right)^{n}}{n!}=-\log\left[1-\sum_{n\ge1}\left(\frac{a}{b-a}n\right)_{n-1}\frac{\left(x^{\frac{a}{b-a}}-x^{\frac{b}{b-a}}\right)^{n}}{n!}\right].
\]
The left hand-side is identified from (\ref{eq:Fab}) as $\left(b-a\right)g_{a,b}\left(x^{\frac{1}{b-a}}\right)$
and we deduce
\[
\left(b-a\right)g_{a,b}\left(x^{\frac{1}{b-a}}\right)=-\log\left[1-\sum_{n\ge1}\left(\frac{a}{b-a}n\right)_{n-1}\frac{\left(x^{\frac{a}{b-a}}-x^{\frac{b}{b-a}}\right)^{n}}{n!}\right]
\]
or equivalently 
\[
g_{a,b}^{b-a}\left(x\right)=1-\sum_{n\ge1}\left(\frac{a}{b-a}n\right)_{n-1}\frac{\left(x^{a}-x^{b}\right)^{n}}{n!}.
\]

\subsection{Proof of Proposition \ref{prop:Prop10}}

Start from
\begin{align*}
\int_{0}^{1}g_{a,b}^{b-a}\left(x\right)x^{b-a-1}dx & =\int_{0}^{1}x^{b-a-1}\left(1-\sum_{n\ge1}\left(\frac{a}{b-a}n\right)_{n-1}\frac{\left(x^{a}-x^{b}\right)^{n}}{n!}\right)\\
 & =\frac{1}{b-a}-\sum_{n\ge1}\frac{\left(\frac{a}{b-a}n\right)_{n-1}}{n!}\int_{0}^{1}x^{b-a-1}\left(x^{a}-x^{b}\right)^{n}dx.
\end{align*}
The integral is a beta integral evaluated as
\[
\int_{0}^{1}x^{b-a-1}\left(x^{a}-x^{b}\right)^{n}dx=\frac{1}{b-a}\frac{\Gamma\left(\frac{a}{b-a}n+1\right)\Gamma\left(n+1\right)}{\Gamma\left(\frac{a}{b-a}n+n+2\right)}
\]
so that the sum is
\begin{align*}
\sum_{n\ge1}\left(\frac{a}{b-a}n\right)_{n-1}\int_{0}^{1}x^{b-a-1}\frac{\left(x^{a}-x^{b}\right)^{n}}{n!}dx & =\frac{1}{b-a}\sum_{n\ge1}\frac{\Gamma\left(\frac{a}{b-a}n+n-1\right)}{\Gamma\left(\frac{a}{b-a}n\right)}\frac{\Gamma\left(\frac{a}{b-a}n+1\right)\Gamma\left(n+1\right)}{n!\Gamma\left(\frac{a}{b-a}n+n+2\right)}\\
 & =\frac{1}{b-a}\frac{a}{b-a}\sum_{n\ge1}n\frac{\Gamma\left(\frac{a}{b-a}n+n-1\right)}{\Gamma\left(\frac{a}{b-a}n+n+2\right)}
\end{align*}
and the latest sum is evaluated as
\[
\sum_{n\ge1}n\frac{\Gamma\left(\frac{a}{b-a}n+n-1\right)}{\Gamma\left(\frac{a}{b-a}n+n+2\right)}=\frac{\left(b-a\right)\left(b+\left(b-a\right)\pi\cot\left(\frac{\pi a}{b}\right)\right)}{2b^{2}}
\]
so that the desired integral is
\[
\frac{1}{b-a}\left[1-\frac{a}{2b^{2}}\left(b+\left(b-a\right)\pi\cot\left(\frac{\pi a}{b}\right)\right)\right].
\]

\subsection{Proof of Proposition \ref{prop:Lambertfunction}}

The critical point is now $x_{0}=e^{-1}$ and the inverse of the right-hand
function is computed as
\[
r^{-1}\left(z\right)=1+\sum_{n\ge1}c_{n}\frac{z^{n}}{n!}
\]
with the coefficients
\[
c_{n}=\lim_{w\to1}\frac{d^{n-1}}{dw^{n-1}}\left(\frac{w-1}{-w\log w}\right)^{n}=-\left(n-1\right)^{n-1},\,\,n\ge1
\]
with the convention $0^{0}=1$ so that $c_{1}=-1.$ We deduce
\begin{align*}
g\left(w\right) & =r^{-1}\left(l\left(w\right)\right)=1-\sum_{n\ge1}\frac{\left(n-1\right)^{n-1}}{n!}\left(-w\log w\right)^{n},\,\,0\le w\le1.
\end{align*}
If we now apply Lagrange's inversion theorem to $\left(\log\circ r\right)^{-1}=r^{-1}\circ\exp,$
we obtain the new coefficients
\[
c_{n}=\lim_{w\to0}\frac{d^{n-1}}{dw^{n}}\left(\frac{w}{-e^{w}w}\right)^{n}=-n^{n-1}
\]
so that
\[
\log g\left(w\right)=-\sum_{n\ge1}\frac{n^{n-1}}{n!}\left(-w\log w\right)^{n}=\sum_{n\ge1}\frac{\left(-n\right)^{n-1}}{n!}\left(w\log w\right)^{n}=W_{0}\left(w\log w\right),
\]
the principal branch of the Lambert function. 

\subsection{Proof of Proposition \ref{prop:Lambertintegrals}}

Using 
\[
\int_{0}^{1}\left(w\log w\right)^{n}dw=\left(-1\right)^{n}\frac{\Gamma\left(n+1\right)}{\left(n+1\right)^{n+1}},\,\,n\ge0,
\]
produces
\begin{align*}
\int_{0}^{1}g\left(w\right)dw & =1-\sum_{n\ge1}\frac{\left(n-1\right)^{n-1}}{\left(n+1\right)^{n+1}}=0.728466.
\end{align*}

Moreover, since $\int_{0}^{1}w^{n-1}\left(\log w\right)^{n}dw=\frac{\left(-1\right)^{n}}{n^{n}}\left(n-1\right)!,$
we deduce
\begin{align*}
\int_{0}^{1}\frac{-\log g\left(w\right)}{w}dw & =\int_{0}^{1}\frac{-W_{0}\left(w\log w\right)}{w}dw=\sum_{n\ge1}\frac{\left(-n\right)^{n-1}}{n!}\int_{0}^{1}w^{n-1}\left(\log w\right)^{n}dw\\
 & =\sum_{n\ge1}\frac{\left(-n\right)^{n}}{n!}\frac{\left(-1\right)^{n}}{n^{n}}\left(n-1\right)!=\sum_{n\ge1}\frac{1}{n^{2}}=\zeta\left(2\right).
\end{align*}
We now use the formula
\[
\int_{0}^{1}h\left(f_{a}\left(x\right)\right)h'\left(x\right)dx=2\int_{0}^{1}h\left(g_{a}\left(x\right)\right)h'\left(x\right)dx-h\left(1\right)\left(h\left(1\right)-h\left(f\left(1\right)\right)\right)
\]
to deduce
\[
\int_{0}^{1}\frac{-\log f\left(x\right)}{x}dx=2\int_{0}^{1}\frac{-\log\left(g\left(x\right)\right)}{x}-h^{2}\left(1\right)+h\left(1\right)h\left(f\left(1\right)\right)
\]
with $h\left(x\right)=-x\log x$ and $h\left(1\right)=0,\,\,h\left(f\left(1\right)\right)=\lim_{x\to1}h\left(f\left(x\right)\right)=\lim_{x\to0}-x\log x=0$
so that
\[
\int_{0}^{1}\frac{-\log f\left(x\right)}{x}dx=2\int_{0}^{1}\frac{-\log\left(g\left(x\right)\right)}{x}=2\zeta\left(2\right)=\frac{\pi^{2}}{3}.
\]

\end{document}